\documentclass[12pt]{amsart}      
\usepackage{txfonts}
\usepackage{amssymb}
\usepackage{eucal}
\usepackage{graphicx}
\usepackage{amsmath}
\usepackage{amscd}
\usepackage[all]{xy}           
\usepackage[active]{srcltx} 

\usepackage{amsfonts,latexsym}
\usepackage{xspace}
\usepackage{epsfig}
\usepackage{float}
\usepackage{color}
\usepackage{fancybox}
\usepackage{colordvi}
\usepackage{multicol}
\usepackage{colordvi}
\usepackage[hypertex]{hyperref} 

\newtheorem{theorem}{Theorem}[section]
\newtheorem{proposition}[theorem]{Proposition}

\newtheorem{corollary}[theorem]{Corollary}
\newtheorem{prop-def}{Proposition-Definition}[section]
\newtheorem{coro-def}{Corollary-Definition}[section]

\newcommand{\nc}{\newcommand}
\nc{\tred}[1]{\textcolor{red}{#1}}
\nc{\tblue}[1]{\textcolor{blue}{#1}}
\nc{\tgreen}[1]{\textcolor{green}{#1}}
\nc{\tpurple}[1]{\textcolor{purple}{#1}}
\nc{\btred}[1]{\textcolor{red}{\bf #1}}
\nc{\btblue}[1]{\textcolor{blue}{\bf #1}}
\nc{\btgreen}[1]{\textcolor{green}{\bf #1}}
\nc{\btpurple}[1]{\textcolor{purple}{\bf #1}}

\renewcommand{\Bbb}{\mathbb}

\newcommand{\efootnote}[1]{}

\renewcommand{\textbf}[1]{}

\newcommand{\delete}[1]{}
\nc{\dfootnote}[1]{{}}          
\nc{\ffootnote}[1]{\dfootnote{#1}}
\nc{\mfootnote}[1]{\footnote{#1}} 
\nc{\ofootnote}[1]{\footnote{\tiny Older version: #1}}

\nc{\mlabel}[1]{\label{#1}}  
\nc{\mcite}[1]{\cite{#1}}  
\nc{\mref}[1]{\ref{#1}}  
\nc{\mbibitem}[1]{\bibitem{#1}} 

\delete{
\nc{\mlabel}[1]{\label{#1}  
{\hfill \hspace{1cm}{\bf{{\ }\hfill(#1)}}}}
\nc{\mcite}[1]{\cite{#1}{{\bf{{\ }(#1)}}}}  
\nc{\mref}[1]{\ref{#1}{{\bf{{\ }(#1)}}}}  
\nc{\mbibitem}[1]{\bibitem[\bf #1]{#1}} 
}

\nc{\mtail}{\leq_t}
\nc{\mhead}{\leq_h}
\nc{\rk}{\mathrm{rk}}
\nc{\mset}[1]{\tilde{#1}}
\nc{\pa}{\frakL}
\nc{\arr}{\rightarrow}
\nc{\lu}[1]{(#1)}
\nc{\mult}{\mrm{mult}}
\nc{\diff}{\mathrm{Der}}
\nc{\indiff}{\mathrm{InDer}}
\nc{\outdiff}{\mathrm{OutDer}}
\nc{\conmat}{connection matrix\xspace}
\nc{\bounmat}{boundary matrix\xspace}
\nc{\pcyc}{\mathfrak c}
\nc{\calpa}{\calp_A}
\nc{\calpal}{\Gamma_{AL}}
\nc{\calpc}{\calp_L}
\nc{\frakDa}{\frakD_1}
\nc{\frakDal}{\frakD_2}
\nc{\frakDc}{\frakD_L}
\nc{\frakDv}{\frakD_V}
\nc{\frakDp}{\frakD_F}
\nc{\frakBa}{\frakB_1}
\nc{\frakBal}{\frakB_2}
\nc{\frakBc}{\frakB_L}
\nc{\frakBv}{\frakB_V}

\nc{\bin}[2]{ (_{\stackrel{\scs{#1}}{\scs{#2}}})}  
\nc{\binc}[2]{ \left (\!\! \begin{array}{c} \scs{#1}\\
    \scs{#2} \end{array}\!\! \right )}  
\nc{\bincc}[2]{  \left ( {\scs{#1} \atop
    \vspace{-1cm}\scs{#2}} \right )}  
\nc{\bs}{\bar{S}}
\nc{\cosum}{\sqsubset}
\nc{\la}{\longrightarrow}
\nc{\rar}{\rightarrow}
\nc{\dar}{\downarrow}
\nc{\dprod}{**}
\nc{\dap}[1]{\downarrow \rlap{$\scriptstyle{#1}$}}
\nc{\md}{\mathrm{dth}}
\nc{\uap}[1]{\uparrow \rlap{$\scriptstyle{#1}$}}
\nc{\defeq}{\stackrel{\rm def}{=}}
\nc{\disp}[1]{\displaystyle{#1}}
\nc{\dotcup}{\ \displaystyle{\bigcup^\bullet}\ }
\nc{\gzeta}{\bar{\zeta}}
\nc{\hcm}{\ \hat{,}\ }
\nc{\hts}{\hat{\otimes}}
\nc{\barot}{{\otimes}}
\nc{\free}[1]{\bar{#1}}
\nc{\uni}[1]{\tilde{#1}}
\nc{\hcirc}{\hat{\circ}}
\nc{\lleft}{[}
\nc{\lright}{]}
\nc{\lc}{\lfloor}
\nc{\rc}{\rfloor}
\nc{\curlyl}{\left \{ \begin{array}{c} {} \\ {} \end{array}
    \right .  \!\!\!\!\!\!\!}
\nc{\curlyr}{ \!\!\!\!\!\!\!
    \left . \begin{array}{c} {} \\ {} \end{array}
    \right \} }
\nc{\longmid}{\left | \begin{array}{c} {} \\ {} \end{array}
    \right . \!\!\!\!\!\!\!}
\nc{\onetree}{\bullet}
\nc{\ora}[1]{\stackrel{#1}{\rar}}
\nc{\ola}[1]{\stackrel{#1}{\la}}
\nc{\ot}{\otimes}
\nc{\mot}{{{\boxtimes\,}}}
\nc{\otm}{\overline{\boxtimes}}
\nc{\sprod}{\bullet}
\nc{\scs}[1]{\scriptstyle{#1}}
\nc{\mrm}[1]{{\rm #1}}
\nc{\margin}[1]{\marginpar{\rm #1}}   
\nc{\dirlim}{\displaystyle{\lim_{\longrightarrow}}\,}
\nc{\invlim}{\displaystyle{\lim_{\longleftarrow}}\,}
\nc{\mvp}{\vspace{0.3cm}}
\nc{\tk}{^{(k)}}
\nc{\tp}{^\prime}
\nc{\ttp}{^{\prime\prime}}
\nc{\svp}{\vspace{2cm}}
\nc{\vp}{\vspace{8cm}}
\nc{\proofbegin}{\noindent{\bf Proof: }}
\nc{\proofend}{$\blacksquare$ \vspace{0.3cm}}
\nc{\modg}[1]{\!<\!\!{#1}\!\!>}
\nc{\intg}[1]{F_C(#1)}
\nc{\lmodg}{\!<\!\!}
\nc{\rmodg}{\!\!>\!}
\nc{\cpi}{\widehat{\Pi}}
\nc{\sha}{{\mbox{\cyr X}}}  
\nc{\shap}{{\mbox{\cyrs X}}} 
\nc{\shpr}{\diamond}    
\nc{\shp}{\ast}
\nc{\shplus}{\shpr^+}
\nc{\shprc}{\shpr_c}    
\nc{\msh}{\ast}
\nc{\zprod}{m_0}
\nc{\oprod}{m_1}
\nc{\vep}{\varepsilon}
\nc{\labs}{\mid\!}
\nc{\rabs}{\!\mid}

\nc{\mmbox}[1]{\mbox{\ #1\ }}
\nc{\fp}{\mrm{FP}} \nc{\rchar}{\mrm{char}} \nc{\End}{\mrm{End}} \nc{\Fil}{\mrm{Fil}}
\nc{\Mor}{Mor\xspace}
\nc{\gmzvs}{gMZV\xspace}
\nc{\gmzv}{gMZV\xspace}
\nc{\mzv}{MZV\xspace}
\nc{\mzvs}{MZVs\xspace}
\nc{\Hom}{\mrm{Hom}} \nc{\id}{\mrm{id}} \nc{\im}{\mrm{im}}
\nc{\incl}{\mrm{incl}} \nc{\map}{\mrm{Map}} \nc{\mchar}{\rm char}
\nc{\nz}{\rm NZ} \nc{\supp}{\mathrm Supp}

\nc{\Alg}{\mathbf{Alg}}
\nc{\Bax}{\mathbf{Bax}}
\nc{\bff}{\mathbf f}
\nc{\bfk}{{\bf k}}
\nc{\bfone}{{\bf 1}}
\nc{\bfx}{\mathbf x}
\nc{\bfy}{\mathbf y}
\nc{\base}[1]{\bfone^{\otimes ({#1}+1)}} 
\nc{\Cat}{\mathbf{Cat}}

\nc{\detail}{\marginpar{\bf More detail}
    \noindent{\bf Need more detail!}
    \svp}
\nc{\Int}{\mathbf{Int}}
\nc{\Mon}{\mathbf{Mon}}
\nc{\rbtm}{{shuffle }}
\nc{\rbto}{{Rota-Baxter }}
\nc{\remarks}{\noindent{\bf Remarks: }}
\nc{\Rings}{\mathbf{Rings}}
\nc{\Sets}{\mathbf{Sets}}

\nc{\BA}{{\Bbb A}} \nc{\CC}{{\Bbb C}} \nc{\DD}{{\Bbb D}}
\nc{\EE}{{\Bbb E}} \nc{\FF}{{\Bbb F}} \nc{\GG}{{\Bbb G}}
\nc{\HH}{{\Bbb H}} \nc{\LL}{{\Bbb L}} \nc{\NN}{{\Bbb N}}
\nc{\KK}{{\Bbb K}} \nc{\QQ}{{\Bbb Q}} \nc{\RR}{{\Bbb R}}
\nc{\TT}{{\Bbb T}} \nc{\VV}{{\Bbb V}} \nc{\ZZ}{{\Bbb Z}}

\nc{\cala}{{\mathcal A}} \nc{\calc}{{\mathcal C}}
\nc{\cald}{{\mathcal D}} \nc{\cale}{{\mathcal E}}
\nc{\calf}{{\mathcal F}} \nc{\calg}{{\mathcal G}}
\nc{\calh}{{\mathcal H}} \nc{\cali}{{\mathcal I}}
\nc{\call}{{\mathcal L}} \nc{\calm}{{\mathcal M}}
\nc{\caln}{{\mathcal N}} \nc{\calo}{{\mathcal O}}
\nc{\calp}{{\mathcal P}} \nc{\calr}{{\mathcal R}}
\nc{\cals}{{\mathcal S}}
\nc{\calt}{{\mathcal T}} \nc{\calw}{{\mathcal W}}
\nc{\calk}{{\mathcal K}} \nc{\calx}{{\mathcal X}}
\nc{\CA}{\mathcal{A}}

\nc{\fraka}{{\mathfrak a}}
\nc{\frakA}{{\mathfrak A}}
\nc{\frakb}{{\mathfrak b}}
\nc{\frakB}{{\mathfrak B}}
\nc{\frakC}{{\mathfrak C}}
\nc{\frakD}{{\mathfrak D}}
\nc{\frakg}{{\mathfrak g}}
\nc{\frakH}{{\mathfrak H}}
\nc{\frakL}{{\mathfrak L}}
\nc{\frakM}{{\mathfrak M}}
\nc{\bfrakM}{\overline{\frakM}}
\nc{\frakm}{{\mathfrak m}}
\nc{\frakP}{{\mathfrak P}}
\nc{\frakN}{{\mathfrak N}}
\nc{\frakp}{{\mathfrak p}}
\nc{\frakR}{{\mathfrak R}}
\nc{\frakS}{{\mathfrak S}}

\font\cyr=wncyr10
\font\cyrs=wncyr7

\begin{document}

\title[A new characterization of hereditary algebras]
{A new characterization of hereditary algebras}

%
%
\author[Yichao Yang]{Yichao Yang}
\address{D\'{e}partement de math\'{e}matiques, Universit\'{e} de Sherbrooke, Sherbrooke, Qu\'{e}bec, Canada, J1K 2R1}
\email{yichao.yang@usherbrooke.ca}

\author[Jinde Xu]{Jinde Xu}
\address{D\'{e}partement de math\'{e}matiques, Universit\'{e} de Sherbrooke, Sherbrooke, Qu\'{e}bec, Canada, J1K 2R1}
\email{jinde.xu@usherbrooke.ca}


\renewcommand{\thefootnote}{\alph{footnote}}
\setcounter{footnote}{-1} \footnote{}
\renewcommand{\thefootnote}{\alph{footnote}}
\setcounter{footnote}{-1} \footnote{\emph{2010 Mathematics Subject
Classification}: 16G20, 16E10.}

\renewcommand{\thefootnote}{\alph{footnote}}
\setcounter{footnote}{-1} \footnote{\emph{Keywords}: Tilting module, $\tau$-tilting module, Hereditary algebras.}


\begin{abstract}
In this short paper we prove that a finite dimensional algebra is hereditary if and only if there is no loop in its ordinary quiver and every $\tau$-tilting module is tilting.
\end{abstract}

\maketitle

\setcounter{section}{0}

\section{Introduction}

Throughout this paper, $A$ stands for a finite dimensional basic algebra over an algebraically closed field $k$ and $Q_{A}$ for its ordinary quiver. The $\tau$-tilting theory was recently introduced by Adachi, Iyama and Reiten in \cite{[AIR]}, which completes the classical tilting theory from the viewpoint of mutation. Note that a tilting $A$-module is always $\tau$-tilting and the converse is true if $A$ is hereditary.

It is therefore interesting to consider whether the heredity of $A$ can be characterized by the property that every $\tau$-tilting $A$-module is tilting. Indeed, it is not true as shown by the following example. Let $A$ be a non-simple local finite dimensional $k$-algebra, then $A$ is the only $\tau$-tilting $A$-module which is also tilting but $A$ is not hereditary. In this counterexample we observe that there exists a loop in $Q_{A}$, thus it is necessary to add some extra restrictions to our assumptions.

Our main result says that a finite dimensional algebra $A$ is hereditary if and only if there is no loop in $Q_{A}$ and every $\tau$-tilting $A$-module is tilting.

\section{Main result}

In what follows, let {\rm mod}-$A$ be the category of finitely generated right $A$-modules. For $M\in$ {\rm mod}-$A$, we denote by {\rm add} $M$ (respectively, {\rm Fac} $M$) the category of all direct summands (respectively, factor modules) of finite direct sums of copies of $M$. Moreover, we denote by $|M|$ the number of pairwise non-isomorphic indecomposable direct summands of $M$. Also, we denote by $P_{i}$ (respectively, $S_{i}$) the indecomposable projective (respectively, simple) $A$-module associated to $i$ for each vertex $i\in (Q_{A})_{0}$. Unless stated otherwise, all modules are assumed to be basic right modules.

Recall that an $A$-module $T$ is \emph{tilting} if it satisfies:

(1)~ {\rm pd}$_{A}T\leq 1$,

(2)~ {\rm Ext}$_{A}^{1}(T,T)=0$ and,

(3)~ there is an exact sequence $0\rightarrow A\rightarrow T_{0}\rightarrow T_{1}\rightarrow 0$ with $T_{0},T_{1}\in$ {\rm add} $T$.

It is well known that the condition (3) can be replaced by (3$^{\prime}$) $|T|=|A|$.

From \cite{[AIR]} we also recall the following definitions in $\tau$-tilting theory.

(1)~ $M$ in {\rm mod}-$A$ is called \emph{$\tau$-rigid} if {\rm Hom}$_{A}(M,\tau M)=0$, where $\tau$ is the Auslander-Reiten translation functor.

(2)~ $M$ in {\rm mod}-$A$ is called \emph{$\tau$-tilting} if $M$ is $\tau$-rigid and $|M|=|A|$.

(3)~ $M$ in {\rm mod}-$A$ is called \emph{support $\tau$-tilting} if there exists an idempotent $e$ of $A$ such that $M$ is a $\tau$-tilting ($A/\langle e\rangle$)-module.

The main result of this paper is the following.

\begin{theorem}\label{main result}
Let $A$ be a finite dimensional algebra. Then $A$ is hereditary if and only if there is no loop in $Q_{A}$ and every $\tau$-tilting $A$-module is tilting.
\end{theorem}

\begin{proof}
First we assume that $A$ is a hereditary algebra. Then for any $A$-module $M$, its projective dimension {\rm pd}$_{A}M\leq 1$. Thus every $\tau$-tilting $A$-module is tilting. Note that $A$ is also a finite dimensional algebra, it is clear that there is no loop in $Q_{A}$.

On the other hand if $A$ is not a hereditary algebra, then its global dimension {\rm gl.dim.}$A>1$. Because the global dimension of $A$ is also equals to the supremum of the set of projective dimensions of all simple $A$-modules, there exists a simple $A$-module $S_{i}$ such that {\rm pd}$_{A}S_{i}>1$. For simplicity, we may take $i=1$.

Since there is no loop in $Q_{A}$, we have {\rm Ext}$_{A}^{1}(S_{1},S_{1})=0$. Moreover, it can easily be seen that {\rm Fac} $S_{1}=${\rm add} $S_{1}$, hence {\rm Ext}$_{A}^{1}(S_{1},${\rm Fac} $S_{1})=0$. Now by [\ref{AS}, Proposition 5.8] we have {\rm Hom}$_{A}(S_{1},\tau S_{1})=0$, therefore the simple module $S_{1}$ is a $\tau$-rigid module.

According to [\ref{AIR}, Theorem 2.10], it follows that $S_{1}$ can be completed to a $\tau$-tilting module, that is, there exists another $A$-module $U$ such that $S_{1}\oplus U$ is a $\tau$-tilting module. On the other hand, we have ${\rm pd}_{A}(S_{1}\oplus U)={\rm max}\{{\rm pd}_{A}S_{1},{\rm pd}_{A}U\}\geq {\rm pd}_{A}S_{1}>1$, hence $S_{1}\oplus U$ is not a tilting module. Finally we find a $\tau$-tilting $A$-module $S_{1}\oplus U$ but not tilting, which contradicts our assumption. The proof of the theorem is now complete.
\end{proof}

From now on we will give some applications of Theorem \ref{main result}. Firstly, the famous \emph{no loop conjecture} stated in \cite{[I]} affirms that the ordinary quiver $Q_{A}$ of $A$ contains no loop if $A$ is of finite global dimension, while the \emph{strong no loop conjecture}, strengthens this to state that a vertex in the ordinary quiver $Q_{A}$ admits no loop if it has finite projective dimension, see [\ref{AIS}, \ref{I}]. Since for a finite dimensional algebra $A$ both these two conjectures have been proved in \cite{[ILP]}, immediately we have the following application.

\begin{corollary}\label{corollary 1}
Let $A$ be a finite dimensional algebra of finite global dimension. Then $A$ is hereditary if and only if every $\tau$-tilting $A$-module is tilting.
\end{corollary}

Recall from \cite{[ASS]} that an $A$-module $M$ is \emph{faithful} if its right annihilator {\rm Ann} $M=\{a\in A | M~a=0\}$ vanishes and $M$ is \emph{sincere} if {\rm Hom}$_{A}(P,M)\neq 0$ for any projective $A$-module $P$. It is easy to see that any faithful module is sincere.

According to [\ref{AIR}, Proposition 2.2] it follows that $\tau$-tilting modules are precisely sincere support $\tau$-tilting modules and tilting modules are precisely faithful support $\tau$-tilting modules, hence to some extent the difference between $\tau$-tilting modules and tilting modules is analog of the difference between faithful modules and sincere modules. Now we have the following direct consequence, which is another application of Theorem \ref{main result}.

\begin{corollary}\label{corollary 2}
Let $A$ be a finite dimensional algebra with no loop in $Q_{A}$. If every sincere $A$-module is faithful, then $A$ is hereditary.
\end{corollary}

However, the converse of Corollary \ref{corollary 2} is not true in general. Let $A$ be the Kronecker algebra and $\mathcal{T}^{A}=\{\mathcal{T}_{\lambda}^{A}\}_{\lambda\in \mathbb{P}_{1}(k)}$ be the $\mathbb{P}_{1}(k)$-family of pairwise orthogonal standard stable tubes, then $A$ is a hereditary algebra and all modules lying on the mouth of the rank one tubes of $\mathcal{T}^{A}$ are sincere $A$-modules but not faithful, see \cite{[SS]}.

Indeed, in [\ref{R}, Corollary 2.3] Ringel have shown the following conclusion. Let $A$ be a hereditary algebra and $M$ be a sincere $A$-module with no self-extensions, then $M$ is also a faithful module. Conversely we have the following result.

\begin{corollary}\label{corollary 3}
Let $A$ be a finite dimensional algebra with no loop in $Q_{A}$. If every sincere $A$-module with no self-extensions is faithful, then $A$ is hereditary.
\end{corollary}

\begin{proof}
By Theorem \ref{main result} it suffices to show that every $\tau$-tilting $A$-module is tilting. Let $M$ be any $\tau$-tilting $A$-module, then by [\ref{AIR}, Proposition 2.2] $M$ is a sincere support $\tau$-tilting module. Thus $M$ is a $\tau$-rigid module, {\rm Hom}$_{A}(M,\tau M)=0$ and hence {\rm Ext}$_{A}^{1}(M,M)=0$, i.e., $M$ has no self-extensions. Now $M$ is a sincere $A$-module with no self-extensions, which is faithful by our assumption. Consequently $M$ is a faithful support $\tau$-tilting module. According to [\ref{AIR}, Proposition 2.2] again it follows that $M$ is actually a tilting module, which completes the proof.
\end{proof}

Using this, the following is now a direct consequence.

\begin{proposition}\label{proposition 1}
Let $A$ be a finite dimensional algebra. Then $A$ is hereditary if and only if there is no loop in $Q_{A}$ and every sincere $A$-module with no self-extensions is faithful.
\end{proposition}

We end with the following easy example.

Let $A$ be the algebra given by quiver $\xymatrix{1\ar[r]^{\alpha} & 2\ar[r]^{\beta} & 3}$ with the relation $\alpha \beta=0$. Then $A$ is not a hereditary algebra and there is no loop in $Q_{A}$. Thus by Theorem \ref{main result} there exists some $\tau$-tilting $A$-modules which are not tilting.

Indeed, it is easy to see that the class of tilting $A$-modules is $\{P_{1}\oplus P_{2}\oplus P_{3}, P_{1}\oplus P_{2}\oplus S_{2}\}$, while the class of $\tau$-tilting $A$-modules is $\{P_{1}\oplus P_{2}\oplus P_{3}, P_{1}\oplus P_{2}\oplus S_{2}\}\cup \{P_{1}\oplus P_{3}\oplus S_{1}\}$ since {\rm pd}$_{A}S_{1}=2>1$. \newline

{\bf Acknowledgements.}~~ The first author is supported by Fonds Qu\'{e}b\'{e}cois de la Recherche sur la Nature et les Technologies (Qu\'{e}bec, Canada) through the Merit Scholarship Program For Foreign Students. This work was carried out when both authors are postdoctoral fellows at Universit\'{e} de Sherbrooke. They would like to thank Prof. Shiping Liu for his helpful discussions and warm hospitality.

\end{document}